\documentclass[11pt]{amsart}
\usepackage{lmodern}
\usepackage{amsmath, amsthm, amssymb, amsfonts}
\usepackage[normalem]{ulem}
\usepackage{hyperref}

\usepackage{verbatim} 
\usepackage{longtable}

\usepackage{mathtools}

\usepackage{tikz}
\usetikzlibrary{decorations.pathmorphing}
\tikzset{snake it/.style={decorate, decoration=snake}}

\usepackage{caption}

\usepackage{tikz-cd}
\usetikzlibrary{arrows}

\theoremstyle{plain}
\newtheorem{thm}{Theorem}[section]
\newtheorem{cor}[thm]{Corollary}
\newtheorem{lem}[thm]{Lemma}
\newtheorem{prop}[thm]{Proposition}
\newtheorem{conj}[thm]{Conjecture}
\newtheorem{question}[thm]{Question}
\newtheorem{speculation}[thm]{Speculation}

\theoremstyle{definition}

\newtheorem{example}[thm]{Example}

\theoremstyle{remark}
\newtheorem{rmk}[thm]{Remark}

\newcommand{\BC}{{\mathbb{C}}}

\newcommand{\BP}{{\mathbb{P}}}
\newcommand{\BQ}{{\mathbb{Q}}}

\newcommand{\BZ}{{\mathbb{Z}}}

\newcommand{\CA}{{\mathcal A}}

\newcommand{\CD}{{\mathcal D}}
\newcommand{\CE}{{\mathcal E}}
\newcommand{\CF}{{\mathcal F}}
\newcommand{\CG}{{\mathcal G}}

\newcommand{\CI}{{\mathcal I}}

\newcommand{\CM}{{\mathcal M}}

\newcommand{\CO}{{\mathcal O}}

\newcommand{\CY}{{\mathcal Y}}

\DeclareFontFamily{OT1}{rsfs}{}
\DeclareFontShape{OT1}{rsfs}{n}{it}{<-> rsfs10}{}
\DeclareMathAlphabet{\curly}{OT1}{rsfs}{n}{it}

\usepackage{tikz}
\usepackage{lmodern}
\usetikzlibrary{decorations.pathmorphing}

\begin{document}
\title[$K3$ categories, cubic fourfolds, and BV filtration]{$K3$ categories, one-cycles on cubic fourfolds, and the Beauville--Voisin filtration}
\date{\today}

\author{Junliang Shen}
\address{ETH Z\"urich, Department of Mathematics}
\email{junliang.shen@math.ethz.ch}

\author{Qizheng Yin}
\address{Peking University, Beijing International Center for Mathematical Research}
\email{qizheng@math.pku.edu.cn}

\begin{abstract}
We explore the connection between $K3$ categories and $0$-cycles on holomorphic symplectic varieties. In this paper, we focus on Kuznetsov's noncommutative $K3$ category associated to a nonsingular cubic $4$-fold.

By introducing a filtration on the $\mathrm{CH}_1$-group of a cubic $4$-fold $Y$, we conjecture a sheaf/cycle correspondence for the associated $K3$ category~$\CA_Y$. This is a noncommutative analog of O'Grady's conjecture concerning derived categories of $K3$ surfaces. We study instances of our conjecture involving rational curves in cubic $4$-folds, and verify the conjecture for sheaves supported on low degree rational curves.

Our method provides systematic constructions of (a) the Beauville--Voisin filtration on the $\mathrm{CH}_0$-group and (b) algebraically coisotropic subvarieties of a holomorphic symplectic variety which is a moduli space of stable objects in $\CA_Y$.
\end{abstract}

\baselineskip=14.5pt
\maketitle

\setcounter{tocdepth}{1} 

\tableofcontents
\setcounter{section}{-1}

\section{Introduction}
\subsection{$K3$ categories and the Beauville--Voisin filtration} The purpose of this paper is to study the interactions between $K3$ categories and the Beauville--Voisin conjecture for holomorphic symplectic varieties. 

A triangulated category is called a $K3$ category if it has the same Serre functor and Hochschild homology as the derived category of coherent sheaves on a $K3$ surface. New examples of $K3$ categories are constructed using the derived categories of certain Fano varieties and semiorthogonal decompositions; see \cite{Kuz, Kuz2, KP}. 

Let $\CA$ be a $K3$ category. If $M$ is a nonsingular projective moduli space of stable objects with respect to a stability condition \cite{Br} on $\CA$, then it is a holomorphic symplectic variety. The nondegenerate holomorphic $2$-form is given by the Serre functor and the Mukai pairing, 
\[
\mathrm{Ext}_\CA^1(\CE, \CE) \times \mathrm{Ext}_\CA^1(\CE, \CE) \rightarrow \mathrm{Ext}_\CA^2(\CE, \CE)\xrightarrow{\mathrm{tr}} \BC.
\]
The Beauville--Voisin conjecture \cite{B2, V, V2} predicts that the Chow group $\mathrm{CH}_0(M)$ admits an increasing filtration 
\begin{equation}\label{BVfil}
S_0\mathrm{CH}_0(M) \subset S_1\mathrm{CH}_0(M) \subset \cdots \subset S_{\frac{1}{2}\dim M}\mathrm{CH}_0(M) = \mathrm{CH}_0(M)
\end{equation}
which is opposite to the conjectural Bloch--Beilinson filtration. Let $K_0(\CA)$ be the Grothendieck group of the triangulated category $\CA$ and let
\begin{equation*}
p_\CA: \CA \rightarrow K_0(\CA)
\end{equation*}
be the natural map. We have the following speculation on the structure~of~$\CA$.

\begin{speculation} \label{spec}
Let $\CA$ be a $K3$ category.
\begin{enumerate}
\item [(a)] There exists an increasing filtration $S_\bullet (\CA)$ on $K_0(\CA)$ which governs the Beauville--Vosin filtration \eqref{BVfil} for any moduli space $M$ as above. More precisely, the $i$-th piece $S_i\mathrm{CH}_0(M)$ is spanned by the classes of $\CE \in M$ with $p_\CA(\CE) \in S_i(\CA)$.

\item[(b)] For every object $\CE \in \CA$, we have
\begin{equation*} \label{generalizedOGconj}
p_\CA(\CE) \in S_{d(\CE)} (\CA)
\end{equation*}
with $d(\CE) = \frac{1}{2}\dim \mathrm{Ext}_\CA^1(\CE, \CE)$.
\end{enumerate}
\end{speculation}

Speculation \ref{spec} (b) can be viewed as a sheaf/cycle correspondence for the $K3$ category $\CA$. For a nonsingular projective moduli space $M$, Speculation~\ref{spec}~(b) implies exactly that
\[S_{\frac{1}{2}\dim M}\mathrm{CH}_0(M) = \mathrm{CH}_0(M).\]

\subsection{O'Grady's conjecture}
The first evidence of Speculation \ref{spec} is the case $\CA = D^b(X)$ where $X$ is a $K3$ surface. In \cite{OG2}, O'Grady introduced a filtration $S_\bullet(X)$ on the Chow group $\mathrm{CH}_0(X)$,
\[
S_0(X) \subset S_1(X) \subset \cdots \subset {S}_i(X) \subset \cdots \subset \mathrm{CH}_0(X).\footnote{
The pull-back via the map $K_0(X) \xrightarrow{c_2} \mathrm{CH}_0(X)$ induces a filtration on $K_0(X)$ as in Speculation \ref{spec} (a).}
\]
Here $S_i(X)$ is the union of $[z]+ \mathbb{Z} \cdot [o_X]$ with $z$ an effective $0$-cycle of length~$i$ and $[o_X] \in \mathrm{CH}_0(X)$ the Beauville--Voisin canonical class \cite{BV}.

The following generalized version of O'Grady's conjecture \cite{OG2} is proven in \cite{SYZ}, based on earlier results of Huybrechts, O'Grady, and Voisin in \cite{Huy, OG2, V1}.

\begin{thm}\label{OGconj}
For any object $\CE \in D^b(X)$, we have
\[
c_2(\CE) \in S_{d(\CE)}(X).
\]
\end{thm}

Theorem \ref{OGconj} established a sheaf/cycle correspondence for $D^b(X)$. Moreover, O'Grady's filtration is indeed expected to govern the Beauville--Voisin filtration for any nonsingular moduli space $M$ of stable objects in $D^b(X)$; see \cite{SYZ} for further details.

\subsection{Cubic fourfolds and one-cycles} \label{onecycle}
In this paper, we discuss Speculation \ref{spec} for $K3$ categories other than the derived categories of $K3$ surfaces. 

Let $Y \subset \mathbb{P}^5$ be a nonsingular cubic hypersurface. Kuznetsov constructed in \cite{Kuz} a $K3$ category $\CA_Y$ as a full subcategory of $D^b(Y)$,
\begin{equation} \label{Kuzcategory}
\CA_Y = \{ \CE\in D^b(Y): \mathrm{Ext}^*_{D^b(Y)}(\CO_Y(i) , \CE) =0 \ \textrm{ for }\ i=0,1,2 \}.
\end{equation}
If $Y$ is very general, then $\CA_Y$ is not equivalent to $D^b(X)$ of a $K3$ surface~$X$.\footnote{In fact, the category $\CA_Y$ for a very general cubic $4$-fold $Y$ is not equivalent to the derived category of twisted sheaves on a $K3$ surface.} Hence $\CA_Y$ is viewed as a noncommutative $K3$ surface. 

Our first result introduces a filtration on the Chow group $\mathrm{CH}_1(Y)$, which serves as a candidate of the filtration in Speculation \ref{spec}.\footnote{Again, the filtration on $K_0(\CA_Y)$ is obtained by pulling back via the natural maps $K_0(\CA_Y) \rightarrow K_0(Y) \xrightarrow{c_3} \mathrm{CH}_1(Y)$.} We briefly describe the construction below; see Section \ref{secfilt} for more details.

Let $F$ denote the Fano variety of lines in $Y$. We fix a uniruled divisor $D$ on~$F$,
\[\begin{tikzcd}
D \arrow[hook]{r}{} \arrow[dashed]{d}{q} & F, \\
B
\end{tikzcd}\]
where $q$ is a rational map whose general fibers are rational curves.

We call a line $l \subset Y$ \emph{special} (with respect to the uniruled divisor $D$) if the~$0$-cycle class $[l] \in \mathrm{CH}_0(F)$ is represented by a point on $D$. A line $l$ is called \emph{canonical} if it satisfies
\[
3[l] = [H]^3 \in \mathrm{CH}_1(Y), 
\]
where $[H] \in \mathrm{CH}^1(Y)$ is the hyperplane class.

We define a filtration $S_\bullet(Y)$ on $\mathrm{CH}_1(Y)$,
\[
S_0(Y) \subset S_1(Y) \subset \cdots \subset {S}_i(Y) \subset \cdots \subset \mathrm{CH}_1(Y),
\]
where $S_i(Y)$ is the union of $[l_1 + l_2 + \dots + l_i] + \mathbb{Z} \cdot [l_0]$ with $l_k$ $(k > 0)$ special lines and $l_0$ a canonical line. It is shown in Lemma \ref{lem1} that the filtration~$S_\bullet(Y)$ does not depend on the choice of $D$ and is ``intrinsic" to $Y$.

We propose the following conjecture relating the $K3$ category $\CA_Y$ to the filtration $S_\bullet(Y)$.

\begin{conj} \label{mainconj}
For any object $\CE \in \CA_Y$, we have 
\[
\mathrm{c}_3(\CE) \in S_{d(\CE)}(Y).
\]
\end{conj}

Here $\mathrm{c}_3$ is the composition of the inclusion $\CA_Y \subset D^b(Y)$ and \[\mathrm{c}_3: D^b(Y) \to \mathrm{CH}_1(Y).\]
See also Remark \ref{rmk2.4} for an equivalent formulation of Conjecture \ref{mainconj}.

Comparing to the derived category of a $K3$ surface, one advantage of studying the $K3$ category $\CA_Y$ is that cubic $4$-folds have a $20$-dimensional moduli space. Hence our filtration provides a candidate of the Beauville--Voisin filtration of certain holomorphic symplectic varieties of $K3^{[n]}$ type in $20$-dimensional families.\footnote{Stability conditions and moduli spaces of stable objects related to $\CA_Y$ are explored in~\cite{BLMS, LLMS}.}

\subsection{Rational curves} \label{secintrorat}
We study the interplay between Conjecture \ref{mainconj} and the geometry of rational curves in nonsingular cubic $4$-folds \cite{BD, dJS, LLMS, LLSS}.

Let 
\[
\iota^\ast: D^b(Y) \rightarrow \CA_Y
\]
be the left adjoint functor of the natural inclusion $\iota_\ast: \CA_Y \hookrightarrow D^b(Y)$. The following theorem concerns low degree rational curves in $Y$.

\begin{thm} \label{rationalcurve}
Let $C \subset Y$ be a nonsingular connected rational curve of degree~\mbox{$\leq 4$}. If $\CE$ is a $1$-dimensional sheaf supported\,\footnote{Here we mean the reduced support of the sheaf $\CE$ is $C$.} on $C$, then Conjecture~\ref{mainconj} holds for~$\iota^\ast \CE$.
\end{thm}

\[
\begin{array}{ c | c | c | c | c }
\deg C &    1 & 2 & 3 & 4 \\ \hline
\min d(\iota^\ast\CE) & 2 & 2 & 4 & 5
\end{array}
\]

For a nonsingular connected rational curve $C$ of degree $\leq 4$, we list in the table above the minimal possible values of
\[
d(\iota^\ast\CE) = \frac{1}{2}\dim \mathrm{Ext}_{\CA_Y}^1{(}\iota^\ast\CE, \iota^\ast\CE{)}
\]
for all $\CE$. These numbers are related to the maximal rationally connected (MRC) fibration on the moduli space of rational curves in $Y$; see Section~\ref{secrat} for further discussions. 

\subsection{Algebraically coisotroptic subvarieties}
Let $M$ be a holomorphic symplectic variety of dimensional $2d$. Following \cite[Definition 0.6]{V2}, a closed subvariety $Z_i \subset M$ of codimension $i$ is called algebraically coisotropic if there exists a diagram 
\[\begin{tikzcd}
Z_i \arrow[hook]{r}{} \arrow[dashed]{d}{q} & M, \\
B_i
\end{tikzcd}\]
such that the general fibers of $q$ are $i$-dimensional, and the restriction of the holomorphic $2$-form on $M$ coincides with the pull-back of a holomorphic $2$-form on $B_i$.

Voisin \cite[Conjecture 0.4]{V2} conjectured that for every $i \leq d$, there exists an algebraically coisotropic subvariety $Z_i \dasharrow B_i$ of codimension $i$ whose general fibers are constant cycle subvarieties of $M$.\footnote{A constant cycle subvariety is a subvariety whose points all share the same class in the $\mathrm{CH}_0$-group of the ambient variety.}

This conjecture was addressed in \cite{SYZ} when $M$ is a moduli space of stable objects in the derived category of a $K3$ surface. We discuss in Section \ref{seccoiso} the connection between Conjecture \ref{mainconj} and Voisin's conjecture for the moduli spaces of stable objects in $\CA_Y$; see Theorem \ref{thm3.2}. The crucial geometric input is the construction in Lemma \ref{lem1.8} of a special uniruled divisor on the Fano variety $F$.

\subsection{Conventions}

Throughout, we work over the complex numbers $\mathbb{C}$. All varieties are assumed to be (quasi-)projective. Morphisms between triangulated categories are $\BC$-linear.

\subsection{Acknowledgements}

We thank Lie Fu, Daniel Huybrechts, Conan Leung, Zhiyuan Li, Davesh Maulik, Georg Oberdieck, Rahul Pandharipande, Jason Starr, Claire Voisin, Chenyang Xu, and Xiaolei Zhao for their interest and for useful discussions. Part of the work presented here was done while both authors were visiting Shanghai Center for Mathematical Sciences in July~2017. 

J.~ S. was supported by grant ERC-2012-AdG-320368-MCSK in the group of Rahul Pandharipande at ETH Z\"urich.

\section{A filtration on \texorpdfstring{$\text{CH}_1(Y)$}{CH\_1(Y)}} \label{secfilt}

Let $Y\subset \BP^5$ be a nonsingular cubic $4$-fold and let $F$ be the Fano variety of lines in $Y$. In this section, we present some basic properties of the filtration $S_\bullet(Y)$ introduced in Section \ref{onecycle}. Our filtration on $\mathrm{CH}_1(Y)$, which is analogous to O'Grady's filtration on the $\mathrm{CH}_0$-group of a $K3$ surface, relies heavily on the geometry of the Fano variety $F$. 

\subsection{Uniruled divisors} \label{secuni}
Uniruled divisors on $F$ play an important role in the definition of the filtration $S_\bullet(Y)$. We first note that there exist uniruled divisors on the Fano variety of lines in any nonsingular cubic $4$-fold. A geometric construction is the following.

In \cite{V0}, Voisin constructed a self-rational map 
\begin{equation}\label{selfmap}
\varphi: F \dasharrow F
\end{equation}
sending a general line $l \subset Y$ to its residual line with respect to the unique plane $\mathbb{P}^2\subset \BP^5$ tangent to $Y$ along $l$. Then the exceptional locus of $\varphi$ gives a uniruled divisor on $F$; see \cite[Proposition 4.4]{V2}. More generally, Charles and Pacienza showed in \cite[Theorem 1.1]{CP} the existence of uniruled divisors on any holomorphic symplectic variety of $K3^{[n]}$ type using moduli spaces of stable maps.

The following lemma asserts that the filtration $S_\bullet(Y)$ does not depend on the choice of the uniruled divisor.

\begin{lem} \label{lem1}
If a line $l \subset Y$ is special with respect to one uniruled divisor $D \subset F$, then it is special with respect to any uniruled divisor of $F$.
\end{lem}

\begin{proof}
We may assume that $D$ is irreducible. Let $D' \subset F$ be another irreducible uniruled divisor. We need to show that every point on $D$ is rationally equivalent to a point on $D'$.

Let $q_{F}$ denote the Beauville--Bogomolov quadratic form on $H^2(F, \BZ)$. From the proof of \cite[Theorem 5.1]{CP}, we see that either
\[
q_F(D, D') \neq 0
\]
or there exists a sequence of irreducible uniruled divisors $D_i$ ($0 \leq i\leq m$) with $D_0 = D$ and $D_m = D'$ satisfying
\[
q_F(D_i, D_{i+1}) \neq 0, \ \ i=0,1,\dots,m-1.
\]
In the first case, by \cite[Lemma 5.2]{CP} the intersection number of every rational curve in the ruling of $D$ and the divisor $D'$ is nonzero. Hence any point on~$D$ is rationally equivalent to a point on $D'$, and Lemma \ref{lem1} follows. In the second case we can use $D_i$ ($1 \leq i\leq m-1$) as transitions.
\end{proof}

\subsection{Zero-cycles on $F$}
We discuss the relationship between the class of a line in $\mathrm{CH}_1(Y)$ and the corresponding point class in $\mathrm{CH}_0(F)$.\footnote{By abuse of notation, for a line $l \subset Y$ we write both $[l] \in \mathrm{CH}_1(Y)$ and $[l] \in \mathrm{CH}_0(F)$.}

Let $P = \{(l,x)\in F \times Y: x\in l\}\subset F\times Y$ be the incidence variety, which induces a morphism
\begin{equation} \label{eqn1}
[P]_\ast : \mathrm{CH}_0(F) \rightarrow \mathrm{CH}_1(Y).
\end{equation}
A result of Paranjape \cite{Par} says that $[P]_\ast$ is surjective. The following fact is noted for later reference.

\begin{lem}\label{lemTF}
The Chow groups $\mathrm{CH}_0(F)$ and $\mathrm{CH}_1(Y)$ are torsion-free.
\end{lem}

\begin{proof}
The statement for $\mathrm{CH}_0(F)$ follows from Ro{\u\i}tman's theorem \cite{Roi}. For~$\mathrm{CH}_1(Y)$, since the morphism $[P]_*$ in \eqref{eqn1} is surjective, it suffices to show that the kernel of $[P]_*$ is divisible. This is done by Shen and Vial in \cite[Theorem 20.5]{SV} and the proof of \cite[Lemma 20.6]{SV}.
\end{proof}

We also show that special lines are sufficient to span $\mathrm{CH}_1(Y)$.

\begin{prop} \label{Cor1.3}
Let $j: D \hookrightarrow F$ be a uniruled divisor. Then $[P]_\ast$ induces a natural isomorphism
\[
\mathrm{Im}(j_\ast: \mathrm{CH}_0(D) \rightarrow \mathrm{CH}_0(F)) \xrightarrow{\sim} \mathrm{CH}_1(Y).
\]
\end{prop}

\begin{proof}
By \cite[Theorem 5.1]{CP}, the image
\[
\mathrm{Im}(j_\ast: \mathrm{CH}_0(D) \rightarrow \mathrm{CH}_0(F)) \subset \mathrm{CH}_0(F)
\]
does not depend on the choice of the uniruled divisor $D\subset F$. Hence we can choose $D$ as the exceptional locus of \eqref{selfmap}. Then \cite[Proposition 19.5 and Theorem 20.5]{SV} imply that
\[
\mathrm{Im}(j_\ast: \mathrm{CH}_0(D) \rightarrow \mathrm{CH}_0(F)) \simeq \mathrm{CH}_0(F)/\mathrm{Ker}([P]_\ast) \simeq \mathrm{CH}_1(Y).
\qedhere\]
\end{proof}

By \cite{SV, V}, the Chow group $\mathrm{CH}_0(F)$ carries a canonical $0$-cycle class $[o_F]$ of degree $1$ which can be taken as any point lying on a constant cycle surface in $F$. Moreover, all $0$-dimensional intersections of divisor classes and Chern classes of $F$ are multiples of $[o_F]$.

Recall that a line $l \subset Y$ is canonical if 
\[
3[l] = [H]^3 \in \mathrm{CH}_1(Y)
\]
where $[H] \in \mathrm{CH}^1(Y)$ is the hyperplane class. The following lemma shows the existence of canonical lines in $Y$ and provides a complete criterion.

\begin{lem} \label{lem1.4}
A line $l \subset Y$ is canonical if and only if
\[
[l] = [o_F] \in \mathrm{CH}_0(F).
\]
\end{lem}

\begin{proof}
By the proof of \cite[Lemma 3.2]{V}, there exists a surface $\Sigma \subset F$ such that the class of every point on $\Sigma$ is $[o_F]\in \mathrm{CH}_0(F)$. We first choose a line $l_0 \subset Y$ lying on $\Sigma \subset F$ such that there exists a plane $\BP_{l_0}^2 \subset \BP^5$ tangent to $Y$ along $l_0$. In particular, we have
\[
[l_0] = [o_F] \in \mathrm{CH}_0(F).
\]

Let $l_0'$ be the residual line of $l_0$ with respect to the plane $\BP_{l_0}^2$,
\[
\BP_{l_0}^2\cdot Y = 2l_0 + l_0'.
\]
By definition, we have $[o_F] = \varphi_\ast ([l_0])=[l_0'] \in \mathrm{CH}_0(F)$. It follows that
\[
[H]^3 = [P]_\ast(2[l_0] + [l_0']) = 3 [P]_\ast [o_F] \in \mathrm{CH}_1(Y).
\]
Hence by Lemma \ref{lemTF}, a line $l \subset Y$ is canonical if and only if 
\begin{equation} \label{eqn1.41}
[P]_\ast [l] = [P]_\ast [o_F] \in \mathrm{CH}_1(Y).
\end{equation}

It suffices to show that \eqref{eqn1.41} is equivalent to $[l]=[o_F] \in \mathrm{CH}_0(F)$. Let 
\[
[l] = [o_F] + [l]_{(2)} + [l]_{(4)} \in \mathrm{CH}_0(F)
\]
be the motivic decomposition of the point class $[l] \in \mathrm{CH}_0(F)$ constructed in \cite[Part 3]{SV}. By \cite[Theorem 20.5]{SV}, the condition \eqref{eqn1.41} is equivalent to $[l]_{(2)} = 0$, which implies $[l] = [o_F]$ after \cite[Theorem 3.4]{SYZ}.
\end{proof}

\begin{example} \label{example}
Let $Y\subset \BP^5$ be a nonsingular cubic $4$-folds which contains a plane. Then there is a uniruled divisor
\[
\begin{tikzcd}
D \arrow[hook]{r}{j} \arrow[dashed]{d}{q} & F \\
X
\end{tikzcd}
\]
over a $K3$ surface $X$; see \cite{Kuz} and \cite[Section 3.2]{SYZ} for the construction. We identify the Chow groups $\mathrm{CH}_0(D)$ and $\mathrm{CH}_0(X)$ via the push-forward $q_\ast$. By \cite[Theorem 3.6]{SYZ}, the embedding $j: D \hookrightarrow F$ induces an injective morphism
\[
j_\ast : \mathrm{CH}_0(X) \simeq \mathrm{CH}_0(D) \hookrightarrow \mathrm{CH}_0(F).
\]
Applying Proposition \ref{Cor1.3}, we find an isomorphism
\begin{equation}\label{eqn10}
[P]_\ast  j_\ast: \mathrm{CH}_0(X) \xrightarrow{\sim} \mathrm{CH}_1(Y).
\end{equation}

We know from Lemma \ref{lem1} that a line $l \subset Y$ is special if and only if the class $[l]\in \mathrm{CH}_1(Y)$ corresponds to a point class $[x] \in \mathrm{CH}_0(X)$ under the isomorphism \eqref{eqn10}. Lemma \ref{lem1.4} and \cite[Theorem 3.6]{SYZ} further imply that a line in $Y$ is canonical if and only if its corresponding point class on $X$ is the Beauville--Voisin class $[o_X] \in \mathrm{CH}_0(X)$. 

In conclusion, our filtration on $\mathrm{CH}_1(Y)$ coincides with O'Grady's filtration on $\mathrm{CH}_0(X)$ under the isomorphism \eqref{eqn10}.
\end{example}

\subsection{Generalities on the filtration $S_\bullet(Y)$}
We prove that $S_\bullet(Y)$ is a filtration into ``cones" for any nonsingular cubic $4$-fold $Y$. This is parallel to \cite[Corollary~1.7]{OG2} in the $K3$ surface case. 

\begin{prop} \label{Prop1.6}
Let $\alpha, \alpha' \in \mathrm{CH}_1(Y)$.
\begin{enumerate}
    \item[(a)] If $\alpha \in S_i(Y)$ and $\alpha' \in S_{i'}(Y)$, then $\alpha + \alpha' \in S_{i+i'}(Y)$.
    \item[(b)] If $\alpha \in S_i(Y)$, then $m\alpha \in S_i(Y)$ for any $m \in \BZ$.
    \item[(c)] We have
    \[
     \bigcup_{i\geq 0} S_i(Y) = \mathrm{CH}_1(Y). 
    \]
\end{enumerate}
\end{prop}

Statement (a) is immediate, and (c) follows from (b) and Proposition \ref{Cor1.3}. The proof of (b) requires the following lemmas.

\begin{lem}\label{lem1.7}
Let $\mathcal{Y} \to T$ be a smooth family of cubic $4$-folds over a nonsingular variety $T$, and let $\alpha \in \mathrm{CH}^3(\CY)$. If the restriction $\alpha|_{\CY_t} \in \mathrm{CH}_1(\CY_t)$ lies in $S_i(\CY_t)$ for a very general point $t \in T$, then the same holds for every point~$t \in T$.
\end{lem}

\begin{proof}
Let $\mathcal{F} \rightarrow T$ be the relative Fano variety of lines associated to the family $\CY \to T$. Since the construction of uniruled divisors in Section \ref{secuni} works universally over the moduli space of nonsingular cubic $4$-folds, we can find a relative uniruled divisor
\begin{equation} \label{diagram}
\begin{tikzcd}
\CD \arrow[hook]{r}{} & \CF \arrow{d}{}  \\
&T
\end{tikzcd}
\end{equation}
whose restriction to every fiber gives a uniruled divisor. 

Let $\CD^{(i)} \rightarrow T$ denote the $i$-th relative symmetric product of $\CD$. Consider the locus
\[
Z=\{\sum_{k=1}^il_{t,k} \in \CD^{(i)}: \alpha|_{\CY_t} = \sum_{k=1}^i[l_{t,k}] + m[l_{t,0}]\in \mathrm{CH}_1(\CY_t) \} \subset \CD^{(i)}
\]
with $l_{t,0} \subset \CY_t$ a canonical line. By the assumption that $\alpha|_{\CY_t}\in S_i(\CY_t)$ for a very general $t\in T$, the locus $Z$ dominates the base $T$. A standard argument using Hilbert schemes shows that $Z$ is a countable union of Zariski closed subsets of $\CD^{(i)}$. Hence there exists a component $Z' \subset Z$ which dominates~$T$ via the natural projection $Z' \rightarrow T$. The restriction of $Z'$ to every fiber of~$\CD^{(i)} \rightarrow T$ represents $\alpha|_{\CY_t}$ as 
\[
\alpha|_{\CY_t}= \sum_{k=1}^i[l_{t,k}] + m[l_{t,0}]\in \mathrm{CH}_1(\CY_t)
\]
with $l_{t,k}$ $(k\geq 1)$ special and $l_{t,0}$ canonical in $\CY_t$. 
\end{proof}

\begin{lem}\label{lem1.8}
Let $Y$ be a general nonsingular cubic $4$-fold and let $F$ be its Fano variety of lines. There exists a uniruled divisor $j:D\hookrightarrow F$ such that for every point $x \in D$ and $m\in \BZ$, we can find $y\in D$ satisfying
\[
m[x] = [y] + \alpha \in \mathrm{CH}_0(D)
\]
with $j_\ast \alpha = (m-1)[o_F] \in \mathrm{CH}_0(F)$.
\end{lem}

\begin{proof}
We first construct the uniruled divisor $D \subset F$.\footnote{We learned this construction from a talk by Kieran O'Grady.} Let \[\check{\mathbb{P}}^5= \mathbb{P}H^0(\BP^5, \CO_{\BP^5}(1))\] be the projective space parametrizing hyperplanes $H \subset \BP^5$. For $1\leq e \leq 5$, let $B_e$ denote the closure of the locus formed by $H \subset \BP^5$ such that the cubic~$3$-fold $H \cap Y$ has $e$ nodes. Since $Y$ is general, the locus $B_e \subset \check{\mathbb{P}}^5$ is nonempty and of codimension $e$. Consider the incidence variety
\[
W=\{(l,H)\subset F\times B_4: l \subset H \cap Y \} \subset F \times B_4,
\]
together with the natural projections
\[\begin{tikzcd}
W \arrow{r}{p} \arrow{d}{q} & F.\\
B_4
\end{tikzcd}\]
Note that the fiber $q^{-1}(H)$ is given by the Fano variety of lines in the cubic $3$-fold $H \cap Y$. 

If a cubic $3$-fold $H \cap Y$ contains a node, a standard fact \cite{GC} says that the Fano variety of lines in $H \cap Y$ is birational to the symmetric product~$C_H^{(2)}$ of a genus $4$ curve $C_H$ formed by lines passing through the node. In our situation, the cubic $3$-fold $H \cap Y$ contains $4$ nodes for every $H\in B_4$, and each extra node creates a node on the curve $C_H$. Hence the fiber $q^{-1}(H)$ is birational to~$C_H^{(2)}$ such that the normalization $E_H$ of the curve $C_H$ has genus~$\leq 1$. It follows that every fiber of $q: W\rightarrow B_4$ is birational to a~$\BP^1$-fibration over $E_H$, and the $3$-fold $W$ is uniruled. We define the uniruled divisor $j:D \hookrightarrow F$ to be the image $p(W) \subset F$.

\medskip
\noindent{\bf Claim.} For any $H \in B_4$, consider the composition
\[
f: q^{-1}(H) \hookrightarrow W \xrightarrow{p} F
\]
which induces a morphism of Chow groups
\[
f_\ast: \mathrm{CH}_0(q^{-1}(H)) \rightarrow \mathrm{CH}_0(F).
\]
Then there exists a point $a_H \in q^{-1}(H)$ such that 
\[
f_\ast [a_H] = [o_F] \in \mathrm{CH}_0(F).
\]

\begin{proof}[Proof of the Claim]
Let $H_0$ be a hyperplane lying in $B_5 \subset \check{\mathbb{P}}^5$. Then the fiber~$q^{-1}(H_0)$ is a rational surface, and the class of every point on~$q^{-1}(H_0)$ is $[o_F]\in \mathrm{CH}_0(F)$; see \cite[Lemma 3.2]{V} or \cite[Proposition 4.5]{V2}.

Let $F_H \subset F$ denote the subvariety of lines contained in the cubic $3$-fold~$H\cap Y$. It suffices to show that 
\begin{equation}\label{Eqn15}
F_H \cap F_{H_0} \neq \emptyset.
\end{equation}
For general hyperplanes $H_1$ and $H_2$, the intersection number
$[F_{H_1}] \cdot [F_{H_2}]$ counts lines in the nonsingular cubic surface $H_1 \cap H_2 \cap Y$. Hence
\[
[F_{H}] \cdot [F_{H_0}] = [F_{H_1}] \cdot [F_{H_2}] = 27,
\]
which proves \eqref{Eqn15}.
\end{proof}

For $H \in B_4$, consider the canonical isomorphism 
\begin{equation}\label{eqn11}
\mathrm{CH}_0(q^{-1}(H)) \simeq \mathrm{CH}_0(E_H^{(2)}).
\end{equation}
By resolution of singularities and the argument of \cite[Lemma 2.2]{SYZ}, any point class $[x] \in \mathrm{CH}_0(q^{-1}(H))$ corresponds to a point class $[x'] \in \mathrm{CH}_0(E_H^{(2)})$ under the isomorphism \eqref{eqn11}. Let $[a'_H] \in \mathrm{CH}_0(E_H^{(2)})$ denote the point class corresponding to $[a_H] \in \mathrm{CH}_0(q^{-1}(H))$ as in the Claim. Since $E_H$ has genus~\mbox{$\leq 1$}, there is an isomorphism
\[
\mathrm{CH}_0(E_H^{(2)}) \simeq \mathrm{CH}_0(E_H).
\]
Then the group law of elliptic curves gives a point $y'\in E_H^{(2)}$ satisfying
\[
m[x'] - [y'] = (m-1)[a'_H] \in \mathrm{CH}_0(E_H^{(2)})
\]
for $x' \in E_H^{(2)}$ and $m\in \BZ$. Again, via the isomorphism \eqref{eqn11}, we find $y\in q^{-1}(H)$ such that
\[
m[x] = [y] + (m-1)[a_H] \in \mathrm{CH}_0(q^{-1}(H)).
\]
This proves the lemma.
\end{proof}

\begin{rmk}
In the argument above, we have constructed a uniruled divisor~$D\subset F$ over an elliptic surface for a general cubic $4$-fold. This special uniruled divisor will also be used in Section \ref{seccoiso} for the connection between Conjecture \ref{mainconj} and Voisin's conjecture \cite[Conjecture 0.4]{V2}.
\end{rmk}

\begin{proof}[Proof of Proposition \ref{Prop1.6} (b)]
It suffices to show that if $l \subset Y$ is special, then for any $m \in \BZ$ there exists a special line $l' \subset Y$ satisfying
\begin{equation}\label{eqn12}
m[l] = [l'] + (m-1)[l_0] \in \mathrm{CH}_1(Y).
\end{equation}
Here $[l_0] \in \mathrm{CH}_1(Y)$ is the class of a canonical line.

First, we consider when $Y$ is general. By Lemma \ref{lem1}, we can assume that the special line $l\subset Y$ lies in the uniruled divisor constructed in Lemma \ref{lem1.8}. Hence there exists a special line $l'$ such that
\begin{equation*}\label{eqn13}
    m[l] = [l'] + (m-1)[o_F] \in \mathrm{CH}_0(F).
\end{equation*}
We deduce \eqref{eqn12} by Lemma \ref{lem1.4} and by applying the correspondence
\[
[P]_\ast: \mathrm{CH}_0(F) \rightarrow \mathrm{CH}_1(Y).
\]

Next, we prove Proposition \ref{Prop1.6} (b) for every nonsingular cubic $4$-fold. We take $T$ to be the moduli space of nonsingular cubic $4$-folds with $\CY \rightarrow T$ the universal family. Consider the relative uniruled divisor $\CD \rightarrow T$ as in \eqref{diagram}. Assume the cubic $4$-fold $Y$ is given by the fiber $\CY_{t_0}$ over $t_0\in T$. A special line $l \subset Y=\CY_{t_0}$ can be chosen from a point lying on the uniruled divisor $\CD_{t_0}$. After taking a finite base change, we may assume that the family $\CD \rightarrow T$ admits a section $s: T \rightarrow \CD$ passing through $l \in \CD_{t_0}$. The section $s$ gives a special line $l_t$ for every cubic $4$-fold $\CY_t$. Since Proposition \ref{Prop1.6} (b) is proven for a general cubic $4$-fold, we have 
\[
m[l_t] \in S_1(\CY_t)
\]
for a general fiber $\CY_t$. Applying Lemma \ref{lem1.7}, we find
\[
m[l]  \in S_1(Y),
\]
which proves \eqref{eqn12}.
\end{proof}

Using the uniruled divisor constructed in Lemma \ref{lem1.8}, we actually obtain the following stronger result.

\begin{prop}\label{prop1.10}
Let $\alpha \in \mathrm{CH}_1(Y)$ and let $m$ be a nonzero integer. We have $\gamma \in S_i(Y)$ if and only if $m\gamma \in S_i(Y)$.
\end{prop}

\begin{proof}
We only need to prove the ``only if" part. By Lemma \ref{lem1.7} and an argument similar to the proof of Proposition \ref{Prop1.6} (b), we may assume $Y$ to be general. Let $l$ be any line lying in the uniruled divisor $D \subset F$ constructed in Lemma \ref{lem1.8}. Then the group law of elliptic curves ensures that there exists a line $l' \in D$ such that 
\[
m[l'] = [l] + (m-1)[o_F] \in \mathrm{CH}_0(F).
\]
This proves the proposition.
\end{proof}

\section{Rational curves in cubic fourfolds} \label{secrat}

Let $Y$ be a nonsingular cubic $4$-fold. In this section, we prove Theorem~\ref{rationalcurve} and discuss its connection to the moduli spaces of rational curves in $Y$.

\subsection{The $K3$ category $\CA_Y$}
The $K3$ category $\CA_Y$ has been introduced by Kuznetsov via the semiorthogonal decomposition of the derived category of a cubic $4$-fold \cite{Kuz, Kuz1, Kuz2}. We first review some basic properties of $\CA_Y$.

Following the notation in \cite{Kuz}, let
\[
D^b(Y) = \langle \CA_Y, \CO_{Y}, \CO_Y(1), \CO_Y(2) \rangle
\]
denote the semiorthogonal decomposition of the derived category $D^b(Y)$ with respect to the exceptional collection $\CO_Y, \CO_Y(1), \CO_Y(2) \in D^b(Y)$. The induced component $\CA_Y$ given by \eqref{Kuzcategory} satisfies the following lemma.

\begin{lem}[{\cite[Section 4]{KM}}] \label{Lem2.1}
Let $\CE, \CF \in \CA_Y$.
\begin{enumerate}
\item[(a)] For $i \geq 3$, we have $\mathrm{Ext}_{D^b(Y)}^i(\CE ,\CF)=0$.
\item[(b)] For $i=0,1,2$, there are canonical isomorphisms 
\[
\mathrm{Ext}_{D^b(Y)}^i(\CE, \CF) \simeq \mathrm{Ext}_{D^b(Y)}^{2-i}(\CF, \CE)^\vee.
\]
\item[(c)] We have
\[
\chi(\CE, \CF) = \sum_{i=0}^2 (-1)^i\dim\mathrm{Ext}_{D^b(Y)}^i(\CE, \CF).
\]
\end{enumerate}
\end{lem}

Let $\CE, \CF \in \CA_Y$. Since \[\mathrm{Ext}_{D^b(Y)}^i(\CE, \CF) = \mathrm{Ext}_{\CA_Y}^i(\CE, \CF),\] 
Lemma \ref{Lem2.1} yields
\begin{equation} \label{eqn17}
2\dim\mathrm{Hom}_{\CA_Y}(\CE, \CE) - \dim\mathrm{Ext}_{\CA_Y}^1(\CE, \CE) = \chi(\CE, \CE).
\end{equation}

The natural inclusion $\iota_\ast : \CA_Y \hookrightarrow D^b(Y)$ admits a left adjoint functor
\[
\iota^\ast: D^b(Y) \rightarrow \CA_Y,
\]
which is the ``projection" from $D^b(Y)$ to the $K3$ category $\CA_Y$.

\begin{lem}\label{Lem2.2}
Let $[H] \in \mathrm{CH}^1(Y)$ be the hyperplane class, and let $[l_0] \in \mathrm{CH}_1(Y)$ be the class of a canonical line. For any $\alpha \in \mathrm{CH}_2(Y)$, we have
\[
[H] \cdot \alpha \in \mathbb{Z} \cdot [l_0] \subset \mathrm{CH}_1(Y).
\]
\end{lem}

\begin{proof}
Consider the following morphisms induced by $j: Y \hookrightarrow \BP^5$,
\[
\mathrm{CH}_2(Y) \xrightarrow{j_\ast} \mathrm{CH}_2(\BP^5) \xrightarrow{j^\ast} \mathrm{CH}_1(Y).
\]
Since $\mathrm{CH}_2(\BP^5) = \BZ \cdot [H]^2$, the class
\[
j^\ast j_\ast \alpha = 3 [H] \cdot \alpha
\]
is proportional to $[H]^3$. Hence the lemma follows from Lemma \ref{lemTF}.
\end{proof}

\begin{cor} \label{cor2.3}
For any $\CE \in D^b(Y)$, we have $c_3(\CE) \in S_i(Y)$ if and only if $c_3(\iota^\ast \CE) \in S_i(Y)$.
\end{cor}

\begin{proof}
Any $\CE \in D^b(Y)$ fits into a distinguished triangle
\begin{equation}\label{eqn18}
\CG \to \CE \to \iota_\ast \iota^\ast \CE \to \CG[1]
\end{equation}
with $\CG \in  \langle \CO_Y, \CO_Y(1), \CO_Y(2)\rangle$. The corollary follows directly from \eqref{eqn18} and Lemma \ref{Lem2.2}.
\end{proof}

\begin{rmk}\label{rmk2.4}
As a consequence of Corollary \ref{cor2.3}, Conjecture \ref{mainconj} is equivalent to the following: for any $\CE \in D^b(Y)$, we have
\[
c_3( \CE) \in S_{d(\iota^\ast \CE)}(Y).
\]
\end{rmk}

Recall the Mukai lattice on $\CA_Y$ introduced in \cite[Section 2]{AT}. Let~$K_{\mathrm{top}}(Y)$ denote the topological $K$-theory \cite{AH} of the cubic $4$-fold $Y$, which is endowed with the Mukai vector 
\[
v: K_{\mathrm{top}}(Y) \rightarrow H^\ast(Y, \BQ)
\]
and the Euler pairing $\chi(-,-)$. The Mukai lattice of $\CA_Y$ is defined to be the abelian group
\[
K_{\mathrm{top}}(\CA_Y) = \{ \kappa \in K_{\mathrm{top}}(Y): \chi([\CO_Y(i)], \kappa)=0 \ \textrm{ for } \ i=0,1,2 \},
\]
to which a weight $2$ Hodge structure is associated; see \cite[Definition 2.2]{AT}.

Let
\[
\mathrm{pr}= \mathrm{pr}_0\circ  \mathrm{pr}_1 \circ  \mathrm{pr}_2: K_{\mathrm{top}}(Y) \to K_{\mathrm{top}}(\CA_Y) 
\]
be the projection map with
\[
\mathrm{pr}_i (\kappa) = \kappa - \chi([\CO_Y(i)], \kappa)\cdot [\CO_Y(i)].
\]
For any $\CE \in D^b(Y)$, we have 
\[
\mathrm{pr}[\CE] = [\iota^\ast \CE] \in K_{\mathrm{top}}(\CA_Y).
\]
We define the Mukai pairing on $K_{\mathrm{top}}(\CA_Y)$ to be the nondegenerate symmetric bilinear form $- \chi(-,-)$, and we write $\kappa^2$ for the self-pairing $(\kappa, \kappa)$. Then~\eqref{eqn17} implies 
\begin{equation} \label{eqn19}
\dim\mathrm{Ext}_{\CA_Y}^1(\CE, \CE) = [\CE]^2 + 2\dim\mathrm{Hom}_{\CA_Y}(\CE, \CE) \geq [\CE]^2+2
\end{equation}
for any $\CE \in \CA_Y$.

Note also that for a line $l \subset Y$, the special classes 
\[
\lambda_i = [\iota^\ast \CO_l(i)] = \mathrm{pr}[\CO_l(i)] \in K_{\mathrm{top}}(\CA_Y), \ \ i=1,2,
\]
span an $A_2$-lattice
\[
A_2=\begin{pmatrix}
2 & -1 \\
-1 & 2 \\
\end{pmatrix}\subset K_{\mathrm{top}}(\CA_Y)
\]
with respect to the Mukai pairing on $K_{\mathrm{top}}(\CA_Y)$.

\subsection{One-dimensional sheaves}
Let $\CE$ be a $1$-dimensional sheaf supported on a nonsingular connected rational curve $C \subset Y$ of degree~$e > 0$. The class~$[\CE] \in K_0(Y)$ can be expressed in terms of line bundles on~$C$. In particular, there exist (uniquely determined) integers $r > 0$ and $m$ such~that
\[
[\CE] = re[O_l(1)] + m[\mathbb{C}_p] \in K_{\mathrm{top}}(Y),
\]
where $\mathbb{C}_p$ is the skyscraper sheaf of a point $p \in Y$. On the other hand, by~\cite[Example 15.3.1]{Ful}, we have
\[
c_3(\CE) = 2r [C] \in \mathrm{CH}_1(Y).
\] 

The following proposition gives the lower bound for 
\[
d(\iota^\ast \CE) =  \frac{1}{2}\dim \mathrm{Ext}_{\CA_Y}^1{(}\iota^\ast\CE, \iota^\ast\CE{)}.
\]

\begin{prop}\label{prop2.5}
With the notation above,
\begin{enumerate}
\item[(a)] if $e=2k$, then
\[
d(\iota^\ast \CE) \geq k^2+1;
\]
\item[(b)] if $e=2k+1$, then
\[
d(\iota^\ast \CE) \geq k^2+k+2.
\]
\end{enumerate}
\end{prop}

Note that the bounds above match the table in Section \ref{secintrorat} for $e\leq 4$. 

\begin{proof}
We have 
\[
\mathrm{pr}[\BC_p] = \lambda_2-\lambda_1 \in K_{\mathrm{top}}(\CA_Y).
\]
Hence
\[
[\iota^\ast \CE] = re[\iota^\ast \CO_l(1)] + m[\iota^\ast \BC_p] = (re-m)\lambda_1 +m\lambda_2 \in K_{\mathrm{top}}(\CA_Y).
\]
By the inequality \eqref{eqn19}, we find
\begin{align*}
2d(\iota^\ast \CE) & \geq [\iota^\ast \CE]^2 +2 \\
                   & = ((re-m)\lambda_1 +m\lambda_2)^2+2\\
                   & = 2(3m^2 -3mre+r^2e^2)+2   
\end{align*}
When $e=2k$, we have 
\[
d(\iota^\ast \CE) \geq 3(m-rk)^2 + (r^2k^2+1) \geq k^2+1.
\]
When $e=2k+1$, we have
\[
d(\iota^\ast \CE) \geq 3(m-rk)(m-rk-1)+ (r^2k^2+rk+2) \geq k^2+k+2. \qedhere
\]
\end{proof}

We write $b(e)$ for the bounds above,
\[
  b(e) =
    \begin{cases}
      k^2+1   & \text{if } e=2k,\\
      k^2+k+2 & \text{if } e=2k+1.
    \end{cases}
\]
To deduce Theorem \ref{rationalcurve}, it suffices to prove the following statement for $e \leq 4$:
\begin{enumerate}
\item[(\dag)] For any nonsingular connected rational curve $C\subset Y$ of degree $e$, we~have
\[
[C] \in S_{b(e)}(Y).
\]
\end{enumerate}
Indeed, assuming ($\dag$) and applying Proposition \ref{Prop1.6} (b), we find
\[
c_3(\CE) = 2r [C] \in S_{b(e)}(Y).
\]
Theorem \ref{rationalcurve} then follows from Proposition \ref{prop2.5} since $d(\iota^\ast \CE) \geq b(e)$.

We prove (\dag) for $e \leq 4$ in Sections \ref{seclct} and \ref{secqua}. In \cite{CS}, it is shown that the moduli space of rational curves of a fixed degree $e$ in $Y$ is irreducible. By Lemma \ref{lem1.7}, the filtration $S_\bullet(Y)$ is preserved under specialization. Hence we only need to consider general rational curves $C \subset Y$.

\subsection{Lines, conics, and twisted cubics} \label{seclct}
Let $g \in \mathrm{CH}^1(F)$ be the polarization class given by the Pl\"ucker embedding of $\mathrm{Gr}(2,6)$. We also fix a uniruled divisor $D\subset F$ in the class $ag$ for some $a>0$. 

\begin{prop}\label{prop2.7}
For a general line $l \subset Y$, there exists a plane $\BP^2_l \subset \BP^5$ and special lines $l_1, l_2 \in D$ such that
\[
\BP^2_l \cdot Y = l + l_1 + l_2.
\]
\end{prop}

\begin{proof}
Given a line $l \subset Y$, we write $S_l$ for the surface in $F$ formed by lines meeting $l$. When $l$ is general, the surface $S_l$ is nonsingular by \cite{V3}. There is an involution 
\[
\tau_l : S_l \to S_l
\]
defined as follows. If $l' \in S_l$ is a line other than $l$, then $\tau_l(l')$ is the residual line of the pair $(l, l')$. If $l' = l$, then $\tau_l(l')= \varphi(l)$. For a point $x\in l$, there is a curve $C_x \subset S_l$ formed by lines passing through $x$. The following intersections on $S_l$ are computed in \cite{V3}:
\begin{equation}\label{eqn22}
[C_x]^2 = [l], \ \ [\tau_l(C_x)]^2 = [\varphi(l)], \ \ g|_{S_l} = [C_x] + 2 [\tau_l(C_x)].
\end{equation}
By \cite[Lemma 18.2]{SV}, the intersection number of $g^2|_{S_l}$ is
\[
g^2|_{S_l} = g^2\cdot [S_l] = 21.
\]
Comparing with \eqref{eqn22}, we find
\[
[C_x] \cdot [\tau_l(C_x)] = 4,
\]
which implies
\[
g|_{S_l} \cdot \tau_{l*}(g|_{S_l}) = {(}[C_x] + 2 [\tau_l(C_x)]{)} \cdot {(} [\tau_l(C_x)] + 2[C_x]{)}= 24.
\]

Consider the curve $D_l \subset S_l$ given by the intersection of $S_l$ and the uniruled divisor $D$. To prove the proposition, it suffices to show that 
\[
D_l \cap \tau_l(D_l) \neq \emptyset.
\]
This is achieved by computing the interesection number
\[
[D_l] \cdot [\tau_l(D_l)] = a^2\cdot {(}g|_{S_l} \cdot \tau_{l*}(g|_{S_l}){)}=24a^2 >0. \qedhere
\]
\end{proof}

Now we prove Theorem \ref{rationalcurve} in degrees $e \leq 3$.

\begin{proof}[Proof of \textup{(\dag)} for $e \leq 3$]
Let $l_0 \subset Y$ be a canonical line. For a general line $l \subset Y$, Proposition \ref{prop2.7} shows the existence of lines $l_1, l_2 \in D$ satisfying
\[
[l] + [l_1] + [l_2] = [H]^3 \in \mathrm{CH}_1(Y).
\]
By Proposition \ref{Prop1.6}, we have
\[
[l] = -[l_1]- [l_2] + 3[l_0] \in S_2(Y).
\]

Next, let $C \subset Y$ be a general conic. Then there is a plane $\BP_C^2 \subset \BP^5$ containing $C$. Let $l$ be the residual line of the conic with respect to the plane $\BP_C^2$,
\[
\BP_C^2 \cdot Y = C + l.
\]
We find
\[
[C] = - [l] + 3[l_0] \in S_2(Y).
\]

Finally, let $C \subset Y$ be a general twisted cubic, which is contained in a unique projective space $\BP_C^3 \subset \BP^5$. The intersection 
\[
Y_C= \BP_C^3 \cap Y\] 
is a nonsingular cubic surface. By \cite[Proposition 4.8]{V2}, there exists a pair of lines $l_1, l_2 \subset Y_C$ such that $C$ lies in the linear system $|\CO_{Y_C}(l_1-l_2)|$. This~yields
\[
[C] = [l_1] -[l_2] + 3[l_0] \in S_4(Y). \qedhere
\]
\end{proof}

\subsection{Quartics and intermediate Jacobians} \label{secqua}
Let $C \subset Y$ be a general quartic rational curve. Then $C$ is contained in a unique hyperplane $H \subset \mathbb{P}^5$, whose intersection with $Y$ is a nonsingular cubic $3$-fold
\[
V = H \cap Y.
\]
The intermediate Jacobian of $V$ is a principally polarized abelian 5-fold
\[
J_V = H^{2,1}(V)^\ast / H_3(V , \BZ).
\]
Let $S$ be the Fano surface of lines in $V$, and let $\mathrm{Alb}(S)$ be the Albanese variety of $S$. By \cite{GC}, the Abel--Jacobi map induces a canonical isomorphism
\begin{equation} \label{eqn25}
\mathrm{Alb}(S) \xrightarrow{\sim} J_V.
\end{equation}

We fix a very ample uniruled divisor $D \subset F$ as in Section \ref{seclct}. Consider the curve $R = D \cap S$ with $R' \rightarrow R$ the normalization. The composition
\[
j: R' \rightarrow R \hookrightarrow S
\]
induces a morphism
\[
u: \mathrm{Jac}(R') \rightarrow \mathrm{Alb}(S),
\]
where $\mathrm{Jac}(R')$ is the Jacobian of the nonsingular curve $R'$.

\begin{lem}\label{Lem2.8}
The morphism $u: \mathrm{Jac}(R') \rightarrow \mathrm{Alb}(S)$ is surjective.
\end{lem}

\begin{proof}
It suffices to show that the morphism
\[
j^\ast : H^1(S, \BQ)   \rightarrow H^1(R',\BQ) 
\]
is injective. Suppose this does not hold. Then the projection formula would imply that the bilinear form
\[
H^1(S, \BQ) \times H^1(S, \BQ) \rightarrow \BQ
\]
\[
\langle \alpha, \beta \rangle = \int_{S} \alpha \cdot \beta \cdot [R]
\]
is degenerate. This contradicts the ampleness of $R$.
\end{proof}

We fix a point $x_0 \in R$ and write $l_{x_0} \subset V$ for the corresponding line. Let~$x_0' \in R'$ be a point in the preimage of $x_0$. For any $k>0$, there is a morphism from the symmetric product $R'^{(k)}$ to $\mathrm{Jac}(R')$ with respect to $x_0'$,
\[
f_k: R'^{(k)} \rightarrow \mathrm{Jac}(R'), \ \  f_k(\sum_i x_i') = \CO_{R'}(\sum_i x'_i-kx_0').
\]

Let 
\[
h_k: R'^{(k)} \rightarrow J_V
\]
be the composition of $f_k: R'^{(k)} \rightarrow \mathrm{Jac}(R')$, $u: \mathrm{Jac}(R') \rightarrow \mathrm{Alb}(S)$, and the isomorphism \eqref{eqn25}.

\begin{cor}\label{cor2.9}
The morphism $h_5$ is surjective.
\end{cor}

\begin{proof}
Let $g$ be the genus of the curve $R'$. Then the morphism 
\begin{equation*}
f_g: R'^{(g)} \rightarrow \mathrm{Jac}(R')
\end{equation*}
is surjective, and Lemma \ref{Lem2.8} implies that $h_g$ is also surjective. In particular, we have $g \geq \dim J_V=5$. 

We show by induction that $h_k$ is surjective for any integer $k$ in the range
\[
5 \leq k \leq g.
\]
The base case is $k = g$. Now assume the surjectivity of $h_{k+1}$. Consider the closed embedding
\begin{equation}\label{eqn26}
R'^{(k)} \hookrightarrow R'^{(k+1)}
\end{equation}
given by $\sum_i x'_i \mapsto \sum_i x'_i + x_0'$.

To show the surjectivity of the composition
\[
h_k: R'^{(k)} \hookrightarrow R'^{(k+1)} \xrightarrow{h_{k+1}} J_V,
\]
it suffices to prove that the divisor $R'^{(k)} \subset R'^{(k+1)}$ in \eqref{eqn26} is ample. 
Let
\[
\sigma_{k+1}: R'^{k+1} \rightarrow R'^{(k+1)}
\]
be the natural quotient map. The pull-back of $O_{R'^{(k+1)}}(R'^{(k)})$ via $\sigma_{k+1}$ is the ample line bundle
\[
\CO_{R'}(x_0') \boxtimes\CO_{R'}(x_0') \boxtimes \cdots \boxtimes\CO_{R'}(x_0').
\]
Since $\pi_{k+1}$ is finite, we obtain the ampleness of $R'^{(k)} \subset R'^{(k+1)}$.
\end{proof}

We finish the proof of Theorem \ref{rationalcurve}.

\begin{proof}[Proof of \textup{(\dag)} for $e= 4$]
First, note that there always exists a canonical line $l_0 \subset Y$ contained in $V$. This can be deduced from the Claim in the proof of Lemma \ref{lem1.8} and specialization.

The Abel--Jacobi map
\[
\mathrm{AJ}: \mathrm{CH}_1(V)_{\mathrm{hom}} \rightarrow J_V
\]
of the cubic $3$-fold $V$ is an isomorphism of abelian groups; see \cite[Theorem~5.6]{Sh} and the references therein. Given the quartic $C \subset V$, consider
\[
\mathrm{AJ}([C]+[l_0]-5[l_{x_0}]) \in J_V.
\]
By Corollary \ref{cor2.9}, there exist $5$ special lines $l_1, \ldots, l_5$ such that
\[
\mathrm{AJ}([C]+[l_0]-5[l_{x_0}]) = \mathrm{AJ} \Big{(}\sum_{i=1}^5 [l_i] - 5[l_{x_0}] \Big{)}.
\]
Hence we have
\[
[C] = \sum_{i=1}^5[l_1] - [l_0] \in S_5(Y). \qedhere
\]
\end{proof}

The argument above essentially proves the following result.

\begin{cor}\label{cor2.10}
For any $\alpha \in \mathrm{CH}_1(Y)$ supported on a general hyperplane section $H \cap Y$, we have
\[
\alpha \in S_5(Y).
\]
\end{cor}

\subsection{Another ten-dimensional example}
Markushevich and Tikhomirov studied in \cite{MT2, MT} the moduli space $\CM_{\mathrm{MT}}$ of rank $2$ vector bundles supported on nonsingular hyperplane sections $H \cap Y$ with $c_1 =0$ and $c_2=2[l]$. 

The (noncompact) moduli space $\CM_{\mathrm{MT}}$ is nonsingular and holomorphic symplectic of dimension $10$. Moreover, every object in $\CM_{\mathrm{MT}}$ lies in the $K3$ category $\CA_Y$ by \cite[Lemma 7.2]{KM}.

As a consequence of Corollary \ref{cor2.10}, we have the the following proposition.

\begin{prop}
Conjecture \ref{mainconj} holds for any $\CE \in \CM_{\mathrm{MT}}$, {i.e.},
\[
c_3(\CE) \in S_5(Y).
\]
\end{prop}

\begin{rmk}
Every object $\CE \in \CM_{\mathrm{MT}}$ is obtained from an extension
\[
0 \to \CO_V \to \CE(H) \to \CI_{E/V}(2H)\to 0,
\]
where $V=H\cap Y$ is a nonsingular hyperplane section and $E$ is a nonsingular quintic elliptic curve. The noncanonical part of $c_3(\CE)$ comes from the $1$-cycle class of $E$.
\end{rmk}

\subsection{Moduli of rational curves}
Theorem \ref{rationalcurve} is closely related to holomorphic symplectic varieties constructed via the moduli spaces of rational curves in a cubic $4$-fold, which we now discuss.

For convenience, we assume $Y$ to be a general cubic $4$-fold. Let $\CM_e$ denote the moduli space of nonsingular connected rational curves of degree $e$ in $Y$. By \cite{CS, dJS}, the variety $\CM_e$ is irreducible of dimension $3e+1$. For $e \leq 4$, there is a MRC fibration
\begin{equation}\label{MRC}
\pi_e:  \CM_e \dashrightarrow \CM'_e
\end{equation}
such that
\begin{enumerate}
\item[(a)] the base $\CM'_e$ is a holomorphic symplectic variety;
\item[(b)] $\dim(\CM'_e)= b(e)$.
\end{enumerate}

We briefly review the geometry of the map \eqref{MRC}. When $e=1$, the variety~$\CM'_1$ is the Fano variety $F$ of lines and \eqref{MRC} is an isomorphism. When~$e=2$, we still have $\CM'_2= F$ and the map \eqref{MRC} sends a conic to its residual line. Hence 
\[
\dim(\CM'_1)=\dim(\CM'_2)=4.
\]

When $e=3$, the map \eqref{MRC} is constructed by Lehn, Lehn, Sorger, and van Straten in \cite{LLSS}, and the holomorphic symplectic $8$-fold $\CM'_3$ is shown in \cite{AL} to be of $K3^{[4]}$ type. Finally, the case $e = 4$ is related to the recent work of Laza, Sacc\`a, and Voisin \cite{LSV, V4}. The variety $\CM'_4$ is a holomorphic symplectic campactification of the (twisted) intermediate Jacobian fibration associated to~$Y$, which is deformation equivalent to O'Grady's $10$-dimensional variety~\cite{OG10}.

In all four cases above, we expect\footnote{This was verified for the Fano variety of lines in \cite{BLMS} and the Lehn--Lehn--Sorger--van Straten $8$-fold in \cite{LLMS}.} that a birational model of the holomorphic symplectic variety $\CM'_e$ can be realized as a moduli space of stable objects in the $K3$ category $\CA_Y$. Furthermore, for a general rational curve~$C\in \CM_e$ with $\CE_C = \pi_e([C]) \in \CA_Y$, there should exist integers $k \neq 0$ and $m$ such that
\begin{equation} \label{111}
c_3(\CE_C) = k[C] + m[l_0] \in \mathrm{CH}_1(Y).
\end{equation}
Here $l_0 \subset Y$ is a canonical line.

Theorem \ref{rationalcurve} says that for $e\leq 4$ and $C\in \CM_e$, we have
\[
[C] \in S_{\frac{1}{2}\dim \CM'_e}(Y),
\]
which is optimal in view of \eqref{111}.

For $e\geq 5$, de Jong and Starr studied in \cite{dJS} the canonical holomorphic~$2$-form on a nonsingular projective model of the moduli space $\CM_e$. Inspired by~\cite[Theorem 1.2]{dJS}, we make the following speculation: for every $e\geq 5$, there exists an algebraically coisotropic subvariety of a holomorphic symplectic variety~$M$, 
\[\begin{tikzcd}
Z \arrow[hook]{r}{j} \arrow[dashed]{d}{q} & M, \\
B
\end{tikzcd}\]
which satisfies a list of properties.
\begin{enumerate}
\item[(a)] The variety $M$ (or its birational model) can be realized as a moduli space of stable objects in $\CA_Y$.
\item[(b)] The general fibers of $q$ are constant cycle subvarieties of $M$.
\item[(c)] For a general point $z\in Z$ with $\CE_z = j(z) \in \CA_Y$, there exists a rational curve $C \in \CM_e$ and integers $k \neq 0$ and $m$ such that
\[
c_3(\CE_z) = k[C]+ m[l_0] \in \mathrm{CH}_1(Y).
\]
Here $l_0 \subset Y$ is a canonical line.
\item[(d)] The dimension of $B$ is $2\mathbf{b}(e)$, where
\[
\mathbf{b}(e)  =
\begin{cases}
\frac{3e}{2} & e \text{ even}, \\
\frac{3e+1}{2}  & e \text{ odd}.
\end{cases}
\]
When $e$ is odd, de Jong and Starr showed that the canonical holomorphic $2$-form on $\CM_e$ is nondegenerate. Hence we expect $B \simeq \CM_e$.
\end{enumerate}

The speculation above, together with Voisin's proposal \cite{V2} and Speculation \ref{spec}, suggests the following optimal bound for the classes of rational curves of degree $\geq 5$ with respect to the filtration $S_\bullet(Y)$.

\begin{conj}
For any nonsingular connected rational curve $C \subset Y$ of degree $e \geq 5$, we have 
\[
[C] \in S_{\mathbf{b}(e)}(Y).
\]
\end{conj}

\begin{rmk}
For $e>5$, the bound $\mathbf{b}(e)$ grows linearly with $e$, and clearly we have
\[
\mathbf{b}(e) < b(e).
\]
Since $\CM_e$ is expected to govern only the point classes on an algebraically coisotropic subvariety in a holomorphic symplectic variety, the quadratic bound $b(e)$ should not be optimal for the classes of curves $C \in \CM_e$.\footnote{When $e=5$, the two bounds $\mathbf{b}(5)$ and $b(5)$ agree. It is possible that $\CM_5$ is birational to a holomorphic symplectic variety.}
\end{rmk}

Indeed, the following proposition provides a (nonoptimal) linear bound for any curve in $Y$.\footnote{We thank Claire Voisin for suggesting this.}

\begin{prop}
For any curve $C \subset Y$ of degree $e$, we have
\[
[C] \in S_{42e}(Y).
\]
\end{prop}

\begin{proof}
We prove that there exist integers $k > 0$ and $m$ such that
\begin{equation}\label{Eqn25}
k[C]= - \big( [l_1 + l_2+ \dots + l_{21e}] \big) + m[l_0] \in \mathrm{CH}_1(Y),
\end{equation}
where $l_i \subset Y$ are lines and $l_0 \subset Y$ is a canonical line. The proposition follows immediately from Proposition \ref{prop1.10} and the $e=1$ case of (\dag). 

Recall that $P = \{(l,x)\in F \times Y: x\in l\}$ is the incidence variety with natural projections
\[
p_F: P\rightarrow F, \ \ p_Y: P \rightarrow Y.
\]
Let $D$ be a nonsingular divisor in the polarization class $g\in \mathrm{CH}^1(F)$. We have the following diagram
\[\begin{tikzcd}
P_D \arrow{r}{f} \arrow{d}{p_D} & Y, \\
D
\end{tikzcd}\]
where $p_D$ is the restriction of $p_F$ to $D\subset F$, and $f$ is the composition of the inclusion $P_D \hookrightarrow P$ and $p_Y$. Then $f$ is a finite morphism such that
\begin{equation}\label{EQN26}
\deg f\cdot [C]= f_\ast f^\ast [C] \in \mathrm{CH}_1(Y).
\end{equation}

Since $P_D$ is a projective bundle over $D$, the class $f^\ast [C]$ can be uniquely expressed as
\begin{equation} \label{eqN27}
f^\ast [C] = p_D^\ast \alpha_0 + p_D^\ast \alpha_1 \cdot f^\ast [H] \in \mathrm{CH}_1(P_D)
\end{equation}
with $\alpha_i \in \mathrm{CH}_i(D)$ and $[H] \in \mathrm{CH}^1(Y)$ the hyperplane class. A direct calculation (as in the proof of \cite[Proposition~6]{BD}) yields
\[
\alpha_0 = -g|_D\cdot \alpha_1, \ \ \alpha_1= p_{D\ast} f^\ast [C].
\]
In particular, we know that $-\alpha_0$ is an effective $0$-cycle class on $D$. Combining~\eqref{EQN26} and \eqref{eqN27}, we find
\begin{equation}\label{eqn123}
\deg f \cdot [C] = f_\ast (p_D^\ast \alpha_0) + (f_\ast p_D^\ast \alpha_1) \cdot [H] \in \mathrm{CH}_1(Y).
\end{equation}
Lemma \ref{Lem2.2} implies that $(f_\ast p_D^\ast \alpha_1) \cdot [H]$ is proportional to the class of a canonical line. The degree of the effective class $-\alpha_0$ is calculated by the intersection number
\[
g|_D \cdot \alpha_1 = g|_D \cdot p_{D*}f^\ast(e[l]) = e(g^2\cdot [S_l]) =21e. 
\]
Here recall that $S_l \subset F$ is the surface formed by lines passing through a given line $l \subset Y$. The last equality above is given by \cite[Lemma 18.2]{SV}. Hence \eqref{eqn123} gives the required expression \eqref{Eqn25}.
\end{proof}

\section{Algebraically coisotropic subvarieties} \label{seccoiso}

Let $M$ be a holomorphic symplectic variety of dimension $2d$. In \cite{V2}, Voisin proposed the following conjecture.\footnote{It is clear that Conjecture \ref{Conj3.1} implies \cite[Conjecture 0.4]{V2}. The converse is proven in \cite[Theorem 1.3]{V2}.}

\begin{conj}[{\cite[Conjecture 0.4]{V2}}] \label{Conj3.1}
For $0 \leq i \leq d$, there is a codimension $i$ algebraically coisotropic subvariety
\[\begin{tikzcd}
Z_i \arrow[hook]{r}{} \arrow[dashed]{d}{q} & M, \\
B_i
\end{tikzcd}\]
such that the general fibers of $q$ are $i$-dimensional constant cycle subvarieties of $M$.
\end{conj}

The following theorem is the main result of this section, which shows that the sheaf/cycle correspondence for $\CA_Y$ can produce algebraically coisotropic varieties as in Conjecture \ref{Conj3.1} for all holomorphic symplectic moduli spaces of stable objects in $\CA_Y$.

\begin{thm} \label{thm3.2}
Conjecture \ref{mainconj} implies Conjecture \ref{Conj3.1} if the holomorphic symplectic variety $M$ is a moduli space of stable objects in $\CA_Y$ for a nonsingular cubic $4$-fold $Y$. 
\end{thm}

\subsection{$K3$ surfaces} \label{seck3case}
Let $X$ be a $K3$ surface. Theorem \ref{thm3.2} is parallel to \cite[Theorem 0.5 (i)]{SYZ} which proves Conjecture \ref{Conj3.1} when $M$ is a moduli space of stable objects in $D^b(X)$. We briefly review the main steps of the proof of \cite[Theorem 0.5 (i)]{SYZ}.

For the moment, assume the holomorphic symplectic variety $M$ is a $2d$-dimensional moduli space of stable objects in $D^b(X)$.

\begin{enumerate}
\item[Step 1.] Let $X^{(d)}$ be the symmetric product. Consider the incidence 
\[
R = \{(\mathcal{E}, \xi) \in M \times X^{(d)} : c_2(\mathcal{E}) = [\xi] + m[o_X] \in \mathrm{CH}_0(X)\},
\]
which is a countable union of Zariski closed subsets of $M \times X^{(d)}$. We denote the natural projections by
\[
p_M: R\rightarrow M, \ \ p_{X^{(d)}}: R \rightarrow X^{(d)}.
\]
By a result by Marian and Zhao \cite{MZ}, all points on the same fiber of~$p_{X^{(d)}}$ (resp.~$p_M$) have the same class in $\mathrm{CH}_0(M)$ (resp.~$\mathrm{CH}_0(X^{(d)})$).
\item[Step 2.] O'Grady's conjecture \cite{OG2}, which was proven in full generality in \cite{SYZ}, implies that both $p_M$ and $p_{X^{(d)}}$ are surjective. Hence we can choose a component $R_0 \subset R$ dominating $M$ and $X^{(d)}$,
\begin{equation} \label{diag}
\begin{tikzcd}
& R_0 \arrow{dl}[swap]{p_M} \arrow{dr}{p_{X^{(d)}}} \\
 M  & &  X^{(d)}.
\end{tikzcd}
\end{equation}
Moreover, both morphisms $p_M$ and $p_{X^{(d)}}$ in the diagram above are generically finite.
\item[Step 3.] For $i\leq d$, the codimension $i$ algebraically coisotropic subvarieties with constant cycle fibers are dense in $X^{(d)}$. Hence we can always find an algebraically coisotropic subvariety $Z\subset X^{(d)}$ such that the morphism $p_{X^{(d)}}$ in \eqref{diag} is generically finite over $Z$, and that the restriction of $p_M$ to $p^{-1}_{X^{(d)}}(Z)$ is also generically finite. Then
\[
Z' = p_M(p^{-1}_{X^{(d)}}(Z))
\]
is a codimension $i$ algebraically coisotropic subvariety of $M$ which satisfies the condition in Conjecture \ref{Conj3.1}.
\end{enumerate}

The main difficulty of the proof of Theorem \ref{thm3.2} is the absence of the $K3$ surface, which breaks down all three steps above. We show how to overcome this issue using the geometry of cubic $4$-folds and their Fano varieties of lines.

\subsection{Step 1} From now on, we take the holomorphic symplectic variety $M$ to be a $2d$-dimensional moduli space of stable objects in the $K3$ category~$\CA_Y$. First, we modify Step 1 in Section \ref{seck3case} by the following construction.

Let $D \subset F$ be a uniruled divisor over a surface $B$,
\[
\begin{tikzcd}
D \arrow[hook]{r}{j} \arrow[dashed]{d}{q} & F. \\
B
\end{tikzcd}
\]
We identify the Chow groups $\mathrm{CH}_0(D)$ and $\mathrm{CH}_0(B)$ via the isomorphism
\[
q_\ast : \mathrm{CH}_0(D) \xrightarrow{\sim} \mathrm{CH}_0(B).
\]
For $k>0$, the embedding $j: D \hookrightarrow F$ induces a morphism
\[
j^{(k)}_\ast: \mathrm{CH}_0(B^{(k)}) \rightarrow \mathrm{CH}_0(F).
\]

We call $W \subset B^{(k)}$ an $F$-constant cycle subvariety if $j^{(k)}_\ast[w]$ is constant in~$\mathrm{CH}_0(F)$ for every point $w\in W$.

\begin{lem}\label{lem3.3}
There is a uniruled divisor $D$ on $F$,
\begin{equation*}
\begin{tikzcd}
D \arrow[hook]{r}{j} \arrow[dashed]{d}{q} & F, \\
B
\end{tikzcd}\end{equation*}
such that the surface $B$ contains infinity many $F$-constant cycle curves $\{C_i\}$ whose union is Zariski dense in $B$.
\end{lem}

\begin{proof}
First, we assume $Y$ to be a general cubic $4$-fold. In the proof of Lemma \ref{lem1.8}, we have constructed a uniruled divisor $q: D \dashrightarrow B$ such that~$B$ admits a fibration
\[
g: B\rightarrow T
\]
whose general fibers are elliptic curves. Moreover, the Claim in the proof of Lemma \ref{lem1.8} implies that there exists a multi-section $C\subset B$ of the morphism~$g$ which is an $F$-constant cycle curve. 

The required density is provided by the torsion structure of the elliptic curves on $B$. More precisely, all irreducible components of the locus
\[
D_i=\{x\in B: i[x-y] =0 \in \mathrm{CH}_0(g^{-1}(t)), y\in C, t\in T\}
\]
give the curves $C_i$ for $i > 0$.

Since specializations preserve uniruled divisors and $F$-constant cycle subvarieties, we obtain the lemma for any cubic $4$-fold. 
\end{proof}

The elliptic surface $B$ in Lemma \ref{lem3.3} plays the role of the $K3$ surface $X$ in Section \ref{seck3case}. Consider the following incidence
\[
R = \{(\mathcal{E}, \xi) \in M \times B^{(d)} : c_3(\mathcal{E}) = [P]_*j^{(d)}_\ast [\xi] + m[l_0] \in \mathrm{CH}_1(Y)\},
\]
where $[P]_*$ is the correspondence in \eqref{eqn1} and $l_0$ is a canonical line. There are the two projections
\[
p_M: R \rightarrow M, \ \ p_{B^{(d)}}: R \rightarrow B^{(d)}. 
\]

\begin{lem}\label{lem3.4}
The cycle class map
\[
\mathrm{cl}_Y: \mathrm{CH}^i(Y)_\BQ \rightarrow H^{2i}(Y, \BQ)
\]
is injective when $i \neq 3$.
\end{lem}
\begin{proof}
The statement for $i = 0$, $1$, and $4$ is immediate, and the $i = 2$ case follows from a standard diagonal decomposition argument. More precisely, let $\Delta_Y \subset Y \times Y$ be the diagonal of $Y$. By the theorem of Bloch and Srinivas~\cite{BS}, there is a decomposition
\begin{equation} \label{eqcor}
[\Delta_Y] = [x \times Y] + [V] \in \mathrm{CH}^4(Y \times Y)_\BQ,
\end{equation}
where $x \in Y$ is a point and $[V]$ is supported on $Y \times W$ for some divisor $W \subset Y$. We apply the correspondences on both sides of \eqref{eqcor} to $\mathrm{CH}^2(Y)_\BQ$. The left-hand side gives back the identity, while the right-hand side factors through the $\mathrm{CH}^1$-group of a desingularization $W'$ of $W$.

The intermediate Jacobian for codimension $2$ cycles on $Y$ is trivial. Given a class $\alpha \in \mathrm{CH}^2(Y)_{\BQ, \mathrm{hom}}$, its image under the right-hand side of \eqref{eqcor} factors through a homologically trivial and Abel--Jacobi trivial class in~$\mathrm{CH}^1(W')$, which is zero.
\end{proof}

The argument in \cite{MZ} gives the following result.

\begin{prop}\label{prop3.5}
Two objects $\CE_1$, $\CE_2 \in M$ satisfy 
\[
[\CE_1] = [\CE_2] \in \mathrm{CH}_0(M)
\]
if and only if
\begin{equation}\label{eqqn}
c_3(\CE_1) = c_3(\CE_2) \in \mathrm{CH}_1(Y).
\end{equation}
\end{prop}

\begin{proof}
By Lemma \ref{lem3.4}, the condition \eqref{eqqn} is equivalent to
\[
\mathrm{ch}(\CE_1) = \mathrm{ch}(\CE_2) \in \mathrm{CH}^\ast(Y)_\BQ.
\]
The rest of the proof is the same as in \cite{MZ} via the (quasi-)universal family over $M \times Y$.
\end{proof}

As a consequence of Proposition \ref{prop3.5}, all points on the same fiber of $p_{B^{(d)}}$ have the same class in $\mathrm{CH}_0(M)$. Moreover, by Proposition \ref{Cor1.3}, every component of a fiber of~$p_{M}$ is an $F$-constant cycle subvariety.

\subsection{Steps 2 and 3}
We modify Steps 2 and 3 in Section \ref{seck3case}, and complete the proof of Theorem \ref{thm3.2}. Conjecture \ref{mainconj} now plays the role of O'Grady's conjecture for $K3$ surfaces.

The following proposition is parallel to \cite[Proposition 1.3]{OG2} and \cite[Corollary 3.4]{V1}.

\begin{prop}
Assuming Conjecture \ref{mainconj}, there is a component $R_0 \subset R$ with the following diagram,
\[
\begin{tikzcd}
& R_0 \arrow{dl}[swap]{p_M} \arrow{dr}{p_{B^{(d)}}} \\
M  & &  B^{(d)},
\end{tikzcd}
\]
such that both morphisms $p_M$ and $p_{B^{(d)}}$ are dominant and generically finite.
\end{prop}

\begin{proof}
Conjecture \ref{mainconj} implies that $R \rightarrow M$ is dominant. Hence we can choose a component $R_0 \subset R$ such that $R_0 \rightarrow M$ is also dominant. Now it suffices to show that the other projection $p_{B^{(d)}}: R_0\rightarrow B^{(d)}$ is also dominant. 

Note that there is a nondegenerate\footnote{Here we mean that the $2$-form is nondegenerate on the regular locus.} $2$-form $\omega_B$ on $B$ satisfying
\[
j^\ast \sigma = q^\ast \omega_B, 
\]
where $\sigma \in H^0(F, \Omega^2_F)$ is the holomorphic symplectic form on $F$. The $2$-form~$\omega_B$ further induces a nondegenerate $2$-form $\omega^{(d)}_B$ on $B^{(d)}$. We only need to prove that the pull-back of $\omega^{(d)}_B$ via $p_{B^{(d)}}: R_0\rightarrow B^{(d)}$ coincides with the pull-back of the holomorphic symplectic form on $M$ via $p_M$ (up to scaling). This is deduced from the fact that fibers of $p_M$ are $F$-constant cycle subvarieties in $B^{(d)}$ and from Mumford's theorem \cite{Mum}.
\end{proof}

This gives the required modification of Step 2. Finally, the density result of Lemma~\ref{lem3.3} plays the role of \cite[Lemma 2.4]{SYZ}, and the proof of Theorem~\ref{thm3.2} is identical to that of \cite[Theorem 0.5 (i)]{SYZ}.


\begin{thebibliography}{10}


\bibitem{AL} N. Addington, M. Lehn, {\em On the symplectic eightfold associated to a Pfaffian cubic fourfold.} J. Reine Angew. Math. 731 (2017), 129--137.

\bibitem{AT} N. Addington, R. P. Thomas, {\em Hodge theory and derived categories of cubic fourfolds.} Duke Math. J. 163 (2014), no. 10, 1885--1927.

\bibitem{AH} M. F. Atiyah, F. Hirzebruch, {\em The Riemann--Roch theorem for analytic embeddings.} Topology 1 (1962), 151--166.




\bibitem{BLMS} A. Bayer, M. Lahoz, E. Macr\'{i}, P. Stellari, {\em Stability conditions on Kuznetsov components.} Appendix joint with X. Zhao. arXiv:1703.10839.

\bibitem{B} A. Beauville, {\em Vari\'{e}t\'{e}s k\"{a}hleriennes dont la premi\`{e}re classe de Chern est nulle.}  J. Differential Geom. 18 (1983), no. 4, 755--782 (1984).

\bibitem{B2} A. Beauville, {\em On the splitting of the Bloch--Beilinson filtration.} Algebraic cycles and motives. Vol. 2, 38--53, London Math. Soc. Lecture Note Ser., 344, Cambridge Univ. Press, Cambridge, 2007.

\bibitem{BD} A. Beauville, R. Donagi, {\em La vari\'{e}t\'{e} des droites d'une hypersurface cubique de dimension $4$.} C. R. Acad. Sci. Paris S\'er. I Math. 301 (1985), no. 14, 703--706.

\bibitem{BV} A. Beauville, C. Voisin, {\em On the Chow ring of a $K3$ surface.} J. Algebraic Geom. 13 (2004), no. 3, 417--426.

\bibitem{BS} S. Bloch, V. Srinivas, {\em Remarks on correspondences and algebraic cycles.} Amer. J. Math. 105 (1983), no. 5, 1235--1253.

\bibitem{Br} T. Bridgeland, {\em Stability conditions on triangulated categories.} Ann. of Math. (2) 166 (2007), no. 2, 317--345.



\bibitem{CP} F. Charles, G. Pacienza, {\em Families of rational curves on holomorphic symplectic varieties and applications to $0$-cycles.} arXiv:1401.4071v2.

\bibitem{GC} C. Clemens, P. Griffiths, {\em The intermediate Jacobian of the cubic threefold.} Ann. of Math. (2) 95 (1972), 281--356.

\bibitem{CS} I. Coskun, J. Starr, {\em Rational curves on smooth cubic hypersurfaces.} Int. Math. Res. Not. IMRN (2009), no. 24, 4626--4641. 

\bibitem{Ful} W. Fulton, {\em Intersection theory.} Second edition. Ergebnisse der Mathematik und ihrer Grenzgebiete. 3. Folge. A Series of Modern Surveys in Mathematics, 2. Springer-Verlag, Berlin, 1998. xiv+470 pp.


\bibitem{Huy} D. Huybrechts, {\em Chow groups of $K3$ surfaces and spherical objects.} J. Eur. Math. Soc. 12 (2010), no. 6, 1533--1551.





\bibitem{dJS} A. J. de Jong, J. Starr, {\em Cubic fourfolds and spaces of rational curves.} Illinois J. Math. 48 (2004), no. 2, 415--450.


\bibitem{Kuz} A. Kuznetsov, {\em Derived categories of cubic fourfolds.} Cohomological and geometric approaches to rationality problems, 219--243, Progr. Math., 282, Birkh\"auser Boston, Inc., Boston, MA, 2010.

\bibitem{Kuz1} A. Kuznetsov, {\em Semiorthogonal decompositions in algebraic geometry.} arXiv:1404.3143.

\bibitem{Kuz2} A. Kuznetsov, {\em Derived categories view on rationality problems.} Rationality problems in algebraic geometry, 67--104, Lecture Notes in Math., 2172, Fond. CIME/CIME Found. Subser., Springer, Cham, 2016.

\bibitem{KM} A. Kuznetsov, D. Markushevich, {\em Symplectic structures on moduli spaces of sheaves via the Atiyah class.} J. Geom. Phys. 59 (2009), no. 7, 843--860.

\bibitem{KP} A. Kuznetsov, A. Perry, {\em Derived categories of Gushel--Mukai varieties.} arXiv:1605.\linebreak06568.

\bibitem{LLMS} M. Lahoz, M. Lehn, E. Macr\`i, P. Stellari, {\em Generalized twisted cubics on a cubic fourfold as a moduli space of stable objects.} arXiv:1609.04573.

\bibitem{LSV} R. Laza, G. Sacca, and C. Voisin, {\em A hyper-K\"ahler compactification of the intermediate Jacobian fibration associated with a cubic $4$-fold.} Acta. Math. 218 (2017), no. 1, 55--135.


\bibitem{LLSS} Ch. Lehn, M. Lehn, Ch. Sorger, D. van Straten, {\em Twisted cubics on cubic fourfolds.} J. Reine Angew. Math. 731 (2017), 87--128.




\bibitem{MZ} A. Marian, X. Zhao, {\em On the group of zero-cycles of holomorphic symplectic varieties.}
arXiv:1711.10045v2.


\bibitem{MT2} D. Markushevich, A. Tikhomirov, {\em The Abel--Jacobi map of a moduli component of vector bundles on the cubic threefold.} J. Algebraic Geom. 10 (2001), no. 1, 37--62.

\bibitem{MT} D. Markushevich, A. Tikhomirov, {\em Symplectic structure on a moduli space of
sheaves on a cubic fourfold.} Izv. Ross. Akad. Nauk Ser. Mat., 67 (2003), 131--158.



\bibitem{Mum} D. Mumford, {\em Rational equivalence of $0$-cycles on surfaces.} J. Math. Kyoto Univ. 9 (1968), 195--204.


\bibitem{OG10} K. G. O'Grady, {\em Desingularized moduli spaces of sheaves on a $K3$.} J. Reine Angew. Math. 512 (1999), 49--117.

\bibitem{OG2} K. G. O'Grady, {\em Moduli of sheaves and the Chow group of $K3$ surfaces.} J. Math. Pures Appl. (9) 100 (2013), no. 5, 701--718.

\bibitem{Par} K. H. Paranjape, {\em Cohomological and cycle-theoretic connectivity.} Ann. of Math. (2) 139 (1994), no. 3, 641--660.


\bibitem{Roi} A. A. Rojtman, {\em The torsion of the group of $0$-cycles modulo rational equivalence.} Ann. of Math. (2) 111 (1980), no. 3, 553--569.


\bibitem{SYZ} J. Shen, Q. Yin, X. Zhao, {\em Derived categories of $K3$ surfaces, O'Grady's filtration, and zero-cycles on holomorphic symplectic varieties.} arXiv: 1705.06953v2.

\bibitem{Sh} M. Shen, {\em On relations among $1$-cycles on cubic hypersurfaces.} J. Algebraic Geom. 23 (2014), no. 3, 539--569. 

\bibitem{SV} M. Shen, C. Vial, {\em The Fourier transform for certain hyperk\"ahler fourfolds.} Mem. Amer. Math. Soc. 240 (2016), no. 1139, vii+163 pp.


\bibitem{V3} C. Voisin, {\em Th\'eor\`eme de Torelli pour les cubiques de $\BP^5$.} Invent. Math. 86 (1986), no. 3, 577--601.

\bibitem{V0} C. Voisin, {\em Intrinsic pseudo-volume forms and $K$-correspondences.} The Fano Conference, 761--792, Univ. Torino, Turin, 2004.

\bibitem{V} C. Voisin, {\em On the Chow ring of certain algebraic hyper-K\"{a}hler manifolds.} Pure Appl. Math. Q. 4 (2008), no. 3, part 2, 613--649.

\bibitem{V1} C. Voisin, {\em Rational equivalence of $0$-cycles on $K3$ surfaces and conjectures of Huybrechts and O'Grady.} Recent advances in algebraic geometry, 422--436, London Math. Soc. Lecture Note Ser., 417, Cambridge Univ. Press, Cambridge, 2015.

\bibitem{V2} C. Voisin, {\em Remarks and questions on coisotropic subvarieties and $0$-cycles of hyper-K\"ahler varieties.} $K3$ surfaces and their moduli, 365--399, Progr. Math., 315, Birkh\"auser/Springer, 2016.

\bibitem{V4} C. Voisin, {\em Hyper-K\"ahler compactification of the intermediate
Jacobian fibration of a cubic fourfold: the twisted case.} arXiv:1611.06679.

\end{thebibliography}
\end{document}